\documentclass[10pt]{article}
\usepackage{amsmath}
\usepackage{amsthm}
\usepackage{amsfonts}
\usepackage{amssymb}
\usepackage{latexsym} 

\usepackage[matrix,tips,graph,curve]{xy}



\makeatletter 
\@addtoreset{equation}{section}
\makeatother

\theoremstyle{plain}
\newtheorem{theorem}[equation]{Theorem}
\newtheorem{corollary}[equation]{Corollary}
\newtheorem{lemma}[equation]{Lemma}
\newtheorem{proposition}[equation]{Proposition}

\theoremstyle{definition}

\newtheorem{remark}[equation]{Remark}

\newcommand{\IC}{\mathbb{C}}

\newcommand{\IR}{\mathbb{R}}

\def\d/{/\mspace{-6.0mu}/}

\newcommand{\db}{\bar{\partial}}

\newcommand{\p}{\partial}

\begin{document}

\title{Geometry of minimal energy Yang--Mills connections}
\author{Mark Stern}

\footnotetext{Duke University, Department of Mathematics;  
e-mail:  stern@math.duke.edu. The author was partially supported by NSF grant DMS-0504890.}
\date{}

\maketitle
\setcounter{section}{0}
\section{Introduction}

Let $G$ be a compact Lie group and $E$ a principal $G-$bundle on a complete oriented Riemannian manifold, $M$. Let $A$ denote a connection on $E$ and $\nabla_A$ the associated covariant derivative on the adjoint bundle, $ad(E)$. The Yang-Mills energy of $A$ is 
$$YM(A):= \|F_A\|^2,$$
where $F_A$ denotes the curvature of $A$.  In four dimensions, $F_A$ decomposes into its self-dual and anti-self-dual components,
$$F_A = F_A^+ + F_A^-,$$
where $F_A^{\pm}$ denotes the projection onto the $\pm 1$ eigenspace of the Hodge star operator. A connection is called self-dual (respectively anti-self-dual) if $F_A = F_A^+$ (respectively $F_A = F_A^-$). A connection is called an {\em instanton} if it is either self-dual or anti-self-dual. An instanton is always a minimizer of the Yang-Mills energy on a compact oriented 4 manifold. This leads to the converse question: in four dimensions are local minima for the Yang-Mills energy necessarily instantons? The answer to this naive question has long been known to be no. (See  \cite{bls} and \cite{bl}). Partial positive results for low rank $G$, however,  were obtained by Bourguignon, Lawson, and Simons in \cite{bls} and \cite{bl}, where they use a variational argument to show that if $G= SU(2)$ or $SU(3)$, and $M$ is a compact oriented 4 dimensional homogeneous space, then the curvature, $F_A$, is self-dual, anti-self-dual, or abelian. In this note, we settle this converse question in four dimensions for nonnegatively curved homogeneous manifolds and offer related weaker results for special geometries in higher dimensions.

Our main result is the following theorem.
\begin{theorem}\label{thm1}
Let $E$ be a principal $G-$bundle on a compact oriented homogeneous Riemannian four manifold, $M$.
Let $A$ be a Yang-Mills minimizing connection on $E$.
The adjoint bundle, $ad(E)$, contains two $\nabla_A$-stable subbundles, $k^+$ and $k^-$, satisfying
$F_A^{\pm}$ is a section of $\Lambda^2T^*M\otimes k^{\pm}$,
$$[k^+,k^-] = 0,$$
and the curvature of $k^+$ is self-dual and that of $k^-$ is anti-self-dual.
If $M$ is assumed to be nonnegatively curved instead of compact, then the same result holds provided that we also assume that for each $a\geq 0$, $\nabla_A^aF_A\in L^2\cap L^4$. 
\end{theorem}

Our proof of Theorem \ref{thm1}  extends the variational argument of Bourguignon, Lawson, and Simons. Let $A_t$ be a smooth family of connections on $E$ with $A_0 = A.$ The assumption that $A$ is a local minimum of the Yang-Mills energy implies the variational inequality
\begin{equation}
\frac{d^2}{dt^2}YM(A_t)_{|t=0} \geq 0.
\end{equation}
The proof relies on choosing useful families of test connections with $A_t-A$ constructed from $F_A$. In \cite{bl}, the test connection $A_t = A+ti_XF_A^+$ was used, where $i_X$ denotes interior multiplication by the vector field $X$, and $X$ runs over a basis of Killing vector fields.  Our results rely on recognizing this variation as only the first term in an infinite family of related variations. 

In the absence of Killing vector fields, the search for natural constructions of test connections leads one to consider special geometries where there exist natural maps 
$\Phi:\bigwedge^2T^*M\otimes ad(E)\rightarrow \bigwedge^1T^*M\otimes ad(E).$ In this case, we can consider variations with $\frac{dA}{dt}(0) = \Phi(F_A)$ and seek additional results.  Covariant constant $3$-forms induce natural maps from $2$-forms to $1$-forms. Hence, one expects new results for $G_2$ manifolds, Calabi-Yau 3 folds, and oriented 3 dimensional manifolds.  We treat the latter two in this note. We do not address the question of existence of minimizing connections in higher dimensions. Simons (see \cite{bl}) proved the nonexistence of nonflat Yang-Mills minimizing connections on $S^n$, $n>4$. This nonexistence result has subsequently been generalized in many directions; see, for example, \cite{kot}, \cite{op},\cite{p}, \cite{sh}, and \cite{x}.

On a Kahler $m-$fold with Kahler form $\omega$ the curvature decomposes as 
$$F_A = F_A^{2,0} + F_{A0}^{1,1} + \frac{1}{m}(\Lambda F_A)\omega + F_A^{0,2},$$
where $\Lambda$ denotes the adjoint of exterior multiplication by $\omega$, 
and $F_{A0}^{1,1} = F_{A}^{1,1} - \frac{1}{m}(\Lambda F_A)\omega.$ 
\begin{theorem}
Let $E$ be a principal $G-$bundle on a complete Calabi-Yau 3 fold. 
Let $A$ be a Yang-Mills minimizing connection on $E$. If $M$ is noncompact, assume further that
$F_A\in L^4$. 
Then $F_A^{0,2}$ takes values in a commutative subbundle of $ad(E)$. 
\end{theorem}

The bulk of our results for $3$ manifolds are presumably well known (see for example \cite[Chapter II, Corollary 2.3]{JT} for the case of $\IR^3$), but we include them here as they fall in the same family of techniques as the preceding results. 

\begin{theorem}Let $E$ be a principal $G-$bundle on a complete three dimensional manifold with nonnegative Ricci curvature.  
Let $A$ be a Yang-Mills minimizing connection on $E$. Then 
$$\nabla_A F_A = 0.$$
 Moreover, $F_A$ takes values in a flat commutative subbundle of $ad(E)$, 
and $F_A=0$ unless $M$ has local flat factors. 
\end{theorem}
We remark that applying the preceding theorem to $S^3$ gives an analytic proof of the triviality of $\pi_2(G)$ 
for all compact Lie groups $G$. 

\noindent{\bf Acknowledgements:}$\,\,\,\,$  We wish to thank Savdeep Sethi for stimulating conversations and for his interest in this work. We also thank Benoit Charbonneau for helpful comments. 
\section{Preliminaries}
Let $M$ be a complete Riemannian manifold and $E$ a principal $G$ bundle over $M$, with $G$ a compact Lie group. 
Let $ad(E)$ denote the adjoint bundle of $E$. Let $A^p(M,ad(E))$ denote the smooth $p-$forms with values in $ad(E)$. Given a connection $A$ on $E$, we denote by $\nabla_A$ the corresponding covariant derivative on $A^{\ast}(M,ad(E))$ induced by $A$ and the Levi-Civita connection of $M$. Let $d_A$ denote the exterior derivative associated to $\nabla_A$ and $F_A$ its curvature. 

We are interested in local minima of the Yang-Mills energy 
$$YM(A) = \|F_A\|^2.$$
Critical points of this energy satisfy the Yang-Mills equation
\begin{equation}d_A^*F_A=0,
\end{equation}
where $d_A^*$ denotes the $(L_2-)$adjoint of $d_A$. 
In addition, all connections satisfy the Bianchi identity
\begin{equation}d_AF_A=0.
\end{equation}
If $A_t$ is a smooth one parameter family of connections then 
\begin{equation}\frac{d}{dt}F_{A_t}=d_{A_t}(\frac{dA}{dt}).
\end{equation}
More generally, if $\psi\in A^1(M,ad(E))$ then 
\begin{equation}\label{expand}F_{A+\psi} = F_A + d_A\psi + \psi\wedge \psi.
\end{equation}
Here we note that our convention on exterior products of $ad(E)$ valued forms is normalized by 
$$(dx^I\otimes v_I)\wedge (dx^J\otimes v_J) = \frac{1}{2}(dx^I\wedge dx^J)\otimes [v_I,v_j].$$
As a notational convenience, we will often denote by $e(w)$ exterior multiplication on the left by a form $w$ (possibly with $ad(E)$ coefficients). Its adjoint is denoted $e^*(w)$. 
Thus 
$$e(w)h := w\wedge h,\,\,\,\text{and}\,\,\,\langle f,e(w)h\rangle = \langle e^*(w)f,h\rangle .$$

If $A$ Minimizes the Yang-Mills energy, then of course it satisfies the inequality 
\begin{equation}\label{taut}\|F_A\|^2 \leq \|F_{A+\psi}\|^2,
\end{equation}
for all smooth compactly supported $\psi$. 
Scaling $\psi$ this leads to the second variation inequality
\begin{equation}\label{secvar}0\leq \|d_A\psi\|^2 + 2\langle F_A,\psi\wedge \psi\rangle.
\end{equation}
\begin{remark}
When considering noncompact manifolds, we may wish to consider variations where $\psi$ is not compactly supported. Let $\eta_j$ be a sequence of functions with $lim_{j\rightarrow\infty}\eta_j = 1$ pointwise and $|d\eta_j|$ uniformly bounded. If we assume merely that $\psi\in C^{1}\cap L^2\cap L^4$, then replacing $\psi$ by $\eta_j\psi$ in (\ref{secvar}) yields $0\leq \|d_A\psi\|^2 + 2\langle F_A,\psi\wedge \psi\rangle$ upon passing to the limit. Hence we may apply this variational inequality to $\psi \in C^{1}\cap L^2\cap L^4$.
\end{remark}
Suppose further that $\Lambda^2(T^*_M)\otimes ad(E)$ decomposes into two orthogonal subbundles 
\begin{equation}\label{decomp}\Lambda^2(T^*_M)\otimes ad(E) = \Lambda^+(E)\oplus\Lambda^-(E),
\end{equation}
such that $\nabla_A$ preserves this decomposition. Let $P^{\pm}$ denote the projection onto these summands. 
We call such a decomposition {\em conservative} if there exists $a,b\in\IR$, not both zero, so that 
\begin{equation}\label{topol}
a\|P^+F_A\|^2 + b\|P^-F_A\|^2 \,\,\text{       is independent of}\,\, A.
\end{equation}

The following elementary lemma clarifies the importance of conservative decompositions. 
\begin{lemma}Given a conservative decomposition, a connection minimizes $YM(A)$ if and only if it minimizes 
 $\|P^-F_A\|^2$ (equivalently, if and only if it minimizes $\|P^+F_A\|^2$).  
\end{lemma}
 
Consequently, for energy minimizing connections and conservative decompositions we have
the additional critical point equation :
\begin{equation}\label{ymc}0 = d_A^*P^-F_A
\end{equation}
and the refined variational inequalities:

\begin{equation}\label{tautc}\|P^-F_A\|^2 \leq \|P^-F_{A+\psi}\|^2,
\end{equation}
and
\begin{equation}\label{secvarc}0\leq \|P^-d_A\psi\|^2 + 2\langle P^-F_A,\psi\wedge \psi\rangle.
\end{equation}
 
The following elementary lemma shows how to begin to extract information about the curvature from the variational inequalities. 
\begin{lemma}\label{standardcomp}Let $\psi\in A^1(M,ad(E))$ satisfy 
$$0 = P^-d_A\psi,\,\,\,\text{and}\,\,\, 0 = \langle P^-F_A,\psi\wedge \psi\rangle.$$
Then 
$$e^*(\psi)P^-F_A = 0.$$
\end{lemma}
\noindent
\begin{proof}
Consider the variation $A_t = A + t\psi + t^pw,$ for $1<p<2$ and $w\in A^1(M,ad(E))$ arbitrary. 
Then expanding (\ref{tautc}) we have 
$$\|P^-F_A\|^2 \leq  \|P^-F_A\|^2 + 2\langle P^-F_A, d_A(t\psi + t^pw) + t^2\psi\wedge \psi + 2t^{p+1}\psi\wedge w\rangle + t^2\|P^-d_A\psi\|^2+O(t^{2p}).$$
Invoking (\ref{ymc}) and our hypotheses on $\psi$, this reduces to 
$$0 \leq 2\langle P^-F_A, 2t^{p+1}\psi\wedge w\rangle +O(t^{2p}+t^3).$$
Replacing $w$ by $-w$, we see that 
$$0 = \langle e^*(\psi)P^-_A,  w\rangle$$
for all $w$, and the lemma follows. \end{proof}

In the following sections we will consider $1-$forms $\psi$ constructed from $F_A$ that satisfy the hypotheses of Lemma \ref{standardcomp} and use them to uncover information about $F_A$ and $A$. 

\section{Dimension $4$: (Anti-)Self-Duality and Homogeneous Spaces}\label{4dim}
In this section we assume that $M$ is a 4 dimensional oriented Riemannian homogeneous space with nonnegative sectional curvature.
Denote the group of isometries of $M$ by $K$ and its Lie algebra by $\mathfrak{k}$.  Identify $\mathfrak{k}$ with the Lie algebra of Killing vector fields on $M$. Fixing a base point $o\in M$ and a metric on $\mathfrak{k}$ induces a decomposition 
$\mathfrak{k} = \mathfrak{p} \oplus \mathfrak{u}$, where $\mathfrak{u}$ is the Lie algebra of the isotropy group of $o$ and therefore also the kernel of the evaluation map $\mathfrak{k}\rightarrow T_oM$. 
Because $K$ is the product of an abelian and a compact group, we may choose the metric on $\mathfrak{k}$ to be invariant under the adjoint action of $K$ and so that for every $x$ the evaluation map 
$\mathfrak{k}\rightarrow T_xM$
is an isometry when restricted to the orthogonal complement of its kernel. 

Let $\{X_j\}_{j=1}^D$ be a basis of $\mathfrak{k}$.
 Let $\phi_{j,t}\colon M\to M$, $j=1,\ldots, D$, $t\in\IR$,  be the associated one parameter families of isometries. 
We define a pullback map 
$$\phi_{j,t}^*: (\Lambda^2T^*_{M}\otimes ad(E))_{\phi_{j,t}(x)} \rightarrow 
(\Lambda^2T^*_M\otimes ad(E))_x$$
by defining the action of $\phi_{j,t}^*$ on the $ad(E)$ factor to be parallel transport along the curve $t\rightarrow \phi_{j,t}(x)$.   Away from a fixed point of $\phi_{j,t}$, we may choose a local frame that is parallel on the integral curves of $X_j$ ($j$ fixed) through all points in a neighborhood of $x$.  In such a frame the connection form, which we also denote $A$, satisfies 
\begin{equation}\label{frame}
i_jA = 0,\text{     and   } i_jF_A = i_jd_AA,
\end{equation}
 where $i_j = i_{X_j}$ denotes interior multiplication by $X_j$. 

Given a local frame $\{s_a\}_a$ for $ad(E)$, we write an $ad(E)$ valued $p-$form $f$ as $\sum_a f^a\otimes s_a.$ Then we have 
$$\phi_{j,t}^*d_A f = \sum_a\phi_{j,t}^*(d f^a)\otimes \phi_{j,t}^* s_a + \sum_{a,b}(-1)^p\phi_{j,t}^* f^a\otimes \phi_{j,t}^*(A_a^bs_b)$$
$$=  d_A \phi_{j,t}^*f + 2(\phi_{j,t}^*A-A)\wedge \phi_{j,t}^*f.$$

In four dimensions (oriented)  we have the decomposition of $\Lambda^2T^*M\otimes ad(E)$ given by the decomposition into self-dual and anti-self-dual summands:
$$\Lambda^2T^*M\otimes ad(E) = (\Lambda^2_+T^*M\otimes ad(E)) \oplus (\Lambda^2_-T^*M\otimes ad(E)).$$
Thus the projections onto the summands are given by 
$$P^{\pm} = \frac{1}{2}(1\pm\ast),$$
where $\ast$ denotes the Hodge star operator. 
Let $p_1(A,E)$ denote the first pontrjagin form of $E$ determined by the connection $A$. Recall that 
$\|F^+_A\|^2 - \|F^-_A\|^2$ 
is a multiple of $\int_M p_1(A,E)$ and is therefore independent of $A$ on compact manifolds. On noncompact manifolds it is constant under variations of $A$ that decay suitably at $\infty$. Hence $P^{\pm}$ define a conservative decomposition. 
Of course for this decomposition, because $d_A^* = -\ast d_A\ast$, we have 
$$d_AP^{\pm}F_A = 0.$$
 Clearly
$$\phi_{j,t}^*P^- = P^-\phi_{j,t}^*.$$
Having lifted the action of $\phi_{j,t}^*$ to $ad(E)$, we obtain an extension of the Lie derivative $L_j=L_{X_j}$ to $ad(E)$ valued forms. It satisfies the usual relation 
$$L_j = d_Ai_j + i_jd_A,$$
and of course
$$L_jf = \frac{d}{du}_{|u=0}\phi_{j,u}^*f.$$
 Then infinitesimally we have 
\begin{equation}\label{invarp}
[L_j,P^-] = 0.\end{equation}
\subsection{Variations}
Set $$F_A^{\pm} = P^{\pm}F_A.$$
The stability results in \cite{bls} and \cite{bl} followed in large part by considering the variations of the connection $A + ti_jF_A^+$. Our theorems in 4 dimensions rely on recognizing these variations as an approximation to the variations $A + i_j\int_0^t\phi_{j,s}^*F_A^+(x)ds.$ Heuristically this variation may be thought of as an attempt to test whether the isometry invariance of the Yang-Mills energy extends to isometry invariance when only a self-dual component of the connection is shifted by the isometry.   

Consider 
\begin{align*}F_{A + i_j\int_0^t\phi_{j,s}^*F_A^+(x)ds} &= F_{A} + d_Ai_j\int_0^t\phi_{j,s}^*F_A^+(x)ds + i_j\int_0^t\phi_{j,u}^*F_A^+(x)\wedge  i_j\int_0^t\phi_{j,s}^*F_A^+(x)duds\\
&= F_{A} + L_j\int_0^t\phi_{j,s}^*F_A^+(x)ds  - 2i_j\int_0^t(A-\phi_{j,s}^*A)\wedge\phi_{j,s}^*F_A^+(x)ds\\
  &\phantom{= F_{A}}+ i_j\int_0^t\phi_{j,u}^*F_A^+(x)\wedge  i_j\int_0^t\phi_{j,s}^*F_A^+(x)duds\\
&= F_{A} + \int_0^t\frac{d}{ds}\phi_{j,s}^*F_A^+(x)ds  + 2i_j\int_0^t\int_0^s\frac{\p}{\p u}\phi_{j,u}^*A  \wedge \phi_{j,s}^*F_A^+(x)duds\\
 &\phantom{= F_{A}}+ i_j\int_0^t\int_0^t\phi_{j,u}^*F_A^+(x)\wedge  i_j\phi_{j,s}^*F_A^+(x)duds.\end{align*}
Using (\ref{frame}), we have 
$$\frac{\p}{\p u}\phi_{j,u}^*A = \phi_{j,u}^*i_jd_AA= \phi_{j,u}^*i_jF_A.$$ 
This and additional manipulations give 
\begin{align*}F_{A + i_j\int_0^t\phi_{j,s}^*F_A^+(x)ds}&= F_{A}^- +  \phi_{j,t}^*F_A^+ \\
 - 2\int_0^t\int_0^s\phi_{j,u}^*i_jF_A  \wedge i_j\phi_{j,s}^*F_A^+(x)duds &\phantom{= }+ \int_0^t\int_0^ti_j\phi_{j,u}^*F_A^+(x)\wedge  i_j\phi_{j,s}^*F_A^+(x)duds \\
&= F_{A}^- + \phi_{j,t}^*F_A^+(x)\\
 - 2\int_0^t\int_0^s\phi_{j,u}^*i_jF_A  \wedge i_j\phi_{j,s}^*F_A^+(x)duds
&\phantom{= }+ \int_0^t\int_0^s\phi_{j,u}^*i_jF_A^+  \wedge i_j\phi_{j,s}^*F_A^+(x)duds\\
 &\phantom{= }+ \int_0^t\int_s^ti_j\phi_{j,u}^*F_A^+(x)\wedge  i_j\phi_{j,s}^*F_A^+(x)duds \end{align*}
Changing the order of integration in the last term and cancelling reduces the preceding to  
\begin{equation}\label{genvar}
F_{A + i_j\int_0^t\phi_{j,s}^*F_A^+ds}= F_{A}^- + \phi_{j,t}^*F_A^+  - 2\int_0^t\int_0^s\phi_{j,u}^*i_jF_A^-  \wedge i_j\phi_{j,s}^*F_A^+duds.
\end{equation}
Set 
$$\Phi_j(t):= \int_0^t\int_0^si_j\phi_{j,u}^*F_A^-(x) \wedge i_j\phi_{j,s}^*F_A^+(x)duds.$$
For later application it is useful to Taylor expand $\Phi_j$.  
We have for all integers $B$, 
\begin{equation}\label{taylor1}
\Phi_j(t)  = \sum_{a,b\geq 0}^{a+b = B}\frac{t^{a+b+2}i_jL_j^aF_A^-(x)  \wedge i_jL_j^bF_A^+(x)}{(a+1)!b!(a+b+2)} +O(t^{B+3}).\end{equation}

With this notation, (\ref{tautc}) becomes 

\begin{equation}\label{vargood}
\|F_A^-\|^2\leq \|F_{A}^-  - 2P^-\Phi_j(t)\|^2,
\end{equation}
or in a more useful form:
\begin{equation}\label{vargood2}
\langle F_{A}^-,  \Phi_j(t)\rangle\leq  \| P^-\Phi_j(t)\|^2 ,
\end{equation}
Equivalently 
 
\begin{equation}\label{vargoodless}
\|F_A^+\|^2 = \|\phi_{j,t}^*F_A^+ \|^2 \leq \|\phi_{j,t}^*F_A^+ 
 - 2P^+\Phi_j(t)\|^2,
\end{equation}

The right hand side of (\ref{vargood2}) is evidently $O(t^4)$, implying the nonpositivity of the $O(t^2)$ terms in the left hand side. The Taylor expansion (\ref{taylor1}) gives for each $j$
\begin{equation}\label{first0}0 \geq \langle F_{A}^-,  i_jF_A^-  \wedge i_jF_A^+\rangle\,\,\,(\text{no }j\text{ sum}).\end{equation}
Switching the roles of $P^-$ and $P^+$, we similarly deduce 
\begin{equation}\label{first00}
0 \geq \langle F_{A}^+,  i_jF_A^+  \wedge i_jF_A^-\rangle\,\,\,(\text{no }j\text{ sum}).
\end{equation}

\begin{lemma}\label{lemma1}Let $f^+$ be a self-dual two form and $f^-$ an anti-self-dual two form. Let $\{e_1,e_2,e_3,e_4\}$ 
be a local orthonormal frame for $TM$.
Then 
$\sum_a i_{e_a}f^+\wedge i_{e_a}f^+$  is  self-dual, and $\sum_a i_{e_a}f^-\wedge i_{e_a}f^-$ is anti-self-dual. If $\phi_1$, $\phi_2$, and $\phi_3$ are $ad(E)$ valued 2 forms, then 
$$\sum_a\langle \phi_1,i_{e_a}\phi_2\wedge i_{e_a}\phi_3 \rangle = 
\sum_a\langle \phi_3,i_{e_a}\phi_1\wedge i_{e_a}\phi_2\rangle.$$
\end{lemma}
\noindent
\begin{proof} This is an elementary computation. \end{proof}

As proved in \cite{bls},\cite{bl}, we now obtain our first commutation result :
\begin{proposition}\label{Law}
$$0 = [F_{st}^+,F_{ij}^-], $$
for all indices $s,t,i,j$. \end{proposition}
\noindent
\begin{proof}
Summing (\ref{first0}), we obtain
$$0 \geq \sum_j\langle F_{A}^-,  i_jF_A^-  \wedge i_jF_A^+\rangle.$$
Because the evaluation map $\mathfrak{k}\rightarrow T_mM$ is an isometry on the orthogonal complement to its kernel, the pointwise inner product, 
$$\sum_j\langle F_{A}^-,  i_jF_A^-  \wedge i_jF_A^+\rangle(m) = \sum_a\langle F_{A}^-,  i_{e_a}F_A^-  \wedge i_{e_a}F_A^+\rangle(m),$$ for a local orthonormal frame $\{e_a\}_a$. Applying (\ref{lemma1}), we see that 
$$\sum_a\langle F_{A}^-,  i_{e_a}F_A^-  \wedge i_{e_a}F_A^+\rangle(m) = 
\sum_a\langle i_{e_a}F_{A}^-\wedge i_{e_a}F_A^-, F_A^+\rangle(m) = 0.$$
Hence each of our inequalities (\ref{first0}) is actually an equality.  
\begin{equation}\label{first0e}0 = \langle F_{A}^-,  i_jF_A^-  \wedge i_jF_A^+\rangle 
= \langle F_{A}^+,  i_jF_A^-  \wedge i_jF_A^-\rangle\,\,\,(\text{no }j\text{ sum}).\end{equation}
Symmetrically we obtain 
\begin{equation}\label{first0o}0 = \langle F_{A}^-,  i_jF_A^+  \wedge i_jF_A^+\rangle\,\,\,(\text{no }j\text{ sum}).\end{equation}
On the other hand, we have 
$$P^-d_A i_jF_A^+ = -P^- i_jd_AF_A^+ + L_jP^- F_A^+ = 0.$$
Hence $\psi = i_jF_A^+$ satisfies the hypotheses of Lemma \ref{standardcomp}, implying 
$$e^*(i_jF_A^+)F_A^- = 0.$$
 Expanding this equality in components, using the duality relations,  and allowing $X_j(m)$ to run over a basis of $T_mM$, we obtain the claimed commutation result:  
$$0 = [F_{st}^+,F_{ij}^-], $$ 
for all indices $s,t,i,j$.
\end{proof}

\subsection{Inductive hypothesis} 
In order to move beyond the commutation of the self-dual with the  anti-self-dual components of the curvature to the construction of $\nabla_A$ stable subbundles $k^{+}$ and $k^-$ of $ad(E)$ with self-dual (respectively anti-self-dual) curvature, we wish to prove $[\nabla_A^iF_A^+,\nabla_A^jF^-_A] = 0$ for all $i$ and $j$. We prove this by induction on $i+j$. 

Denote by $A_N$ the inductive hypothesis:
\begin{equation}\label{AN}
A_N:\qquad \qquad\qquad [\nabla_A^iF_A^+,\nabla_A^jF^-_A] = 0, \text{ for }i+j<N.\end{equation}
We have established $A_1$ in Proposition \ref{Law}. We will show $A_N$ implies $A_{N+1}$. Assume $A_N$ holds, for some $N\geq 1$. In the inductive hypothesis, powers of covariant derivatives can be replaced by powers of Lie derivatives, since they differ by lower order terms. 

Observe that 
\begin{equation}A_N\Rightarrow \Phi_j(t) = O(t^{N+2}).\end{equation} 
Hence $A_N$ and equation (\ref{vargood2}) imply 
\begin{equation}\label{pancake}\langle F_A^-,\Phi_j(t)\rangle \leq   O(t^{2N+4}).\end{equation}
 Set 
$$S_j(t) = \langle F_{A}^-,  \Phi_j(t)\rangle.$$
Then $S(0)= 0$, and 
\begin{equation}\label{sjprime}
S_j'(t) = \langle F_{A}^-,  \int_0^ti_j\phi_{j,u}^*F_A^- \wedge i_j\phi_{j,t}^*F_A^+du\rangle
= \langle F_{A}^-,  \phi_{j,t}^*\int_{-t}^0i_j\phi_{j,u}^*F_A^- \wedge i_jF_A^+du\rangle.\end{equation}
$$S_j'(0)= S_j''(0)= 0.$$

 In order to use the variational inequality (\ref{pancake}) we need to estimate $S_j(t)$. 

Taylor expanding $S_j'$ gives 
$$S_j'(t) =  \langle F_{A}^-, \sum_{a,b\geq 0}^{a+b=m}\frac{t^a}{a!}L_j^a\int_{-t}^0i_j\frac{u^b}{b!}L_j^bF_A^- \wedge i_jF_A^+du\rangle + O(t^{m+2})$$
Hence 
$$S_j(t) =  \langle F_{A}^-, \sum_{a,b\geq 0}^{a+b=m}\frac{(-1)^bt^{a+b+2}}{a!(b+1)!(a+b+2)}L_j^a[i_j L_j^bF_A^- \wedge i_jF_A^+]\rangle + O(t^{m+3}) $$
$$=  \langle F_{A}^-, \sum_{a,b\geq N}^{a+b=m}\frac{(-1)^bt^{a+b+2}}{a!(b+1)!(a+b+2)}L_j^a[i_j L_j^bF_A^- \wedge i_jF_A^+]\rangle + O(t^{m+3}) $$
Hence, using $A_N$ to eliminate lower order terms, (\ref{pancake}) implies
\begin{equation}\label{varinduct}
\langle F_{A}^-, (-1)^NL_j^N[i_j L_j^NF_A^- \wedge i_jF_A^+]\rangle \leq 0.\end{equation}

\begin{lemma}\label{onion}If $A_N$ holds then 
$$\langle F_{A}^-, L_j^N[i_j L_j^NF_A^- \wedge i_jF_A^+]\rangle  = 0.$$
\end{lemma}
\begin{proof}
Set $S(X_j,N) = \langle F_{A}^-, L_j^N[i_j L_j^NF_A^- \wedge i_jF_A^+]\rangle.$ The inequality 
$(-1)^NS(X_j,N)\leq 0$ holds when $X_j$ is replaced by any Killing vector. We will show that the average of $S(X_j,N)$ over the unit sphere of $\mathfrak{k}$ is zero. Hence $S(X_j,N)$ is zero for each $j$ (and each choice of basis of $\mathfrak{k}$). To see this we consider 
$S(\sum_ky^kX_k,N)$ and integrate the resulting degree $2N+2$ homogeneous polynomial in $y$ over the unit sphere. Integration of homogeneous degree $2N+2$ polynomials over the sphere projects onto the span of the radial function, $(r^2)^{N+1}$. Expanding  in a multi-index notation where $L_J = L_{j_1}\cdots L_{j_{|J|}}$, we write  
$$S(\sum_ky^kX_k,N) = \sum_{|I|=N,|J|=N}\sum_{m,p=1}^{dim\mathfrak{k}}y^Iy^Jy^my^p\langle F_{A}^-, L_I[i_m L_{J}F_A^- \wedge i_pF_A^+]\rangle.$$
That integration over the unit sphere projects onto radial functions implies that upon integration of $S$, we are left with a linear combination of coefficients, $\langle F_{A}^-, L_I[i_m L_{J}F_A^- \wedge i_pF_A^+]\rangle$, of $S$ where the indices are contracted pairwise. We will see that all such contractions vanish.  
The condition $A_N$ allows us to replace the Lie derivatives by covariant derivatives, as the difference vanishes in the inner product. We can also drop commutators of derivatives by $A_N$, as all such commutations drop the degree of the differentiation and thus lead to terms which vanish by $A_N$.  We can then use the Yang-Mills equation and $A_N$ to equate to zero all $\sum_ki_kL_kF_A^{\pm}$ terms and $\sum_kL_k^2F_A^{\pm}$ terms. We also invoke Lemma \ref{lemma1} to remove terms with paired indices $m=p$ on the interior products. Thus $m$ and $p$ must both pair with elements of $I$, for if $p$ pairs with an $L_{j_r}$, we can use the Leibniz formula, $A_N$ and the Yang-Mills equation to eliminate the corresponding term. This leaves two $L_j$ terms which must pair with each other, but since 
$$0 = -(d_Ad_A^*+_A^*d_A)F_A^{\pm} = \sum_j L_j^2F_{A}^{\pm}  \,\,\, \text{  (modulo terms vanishing in the inner product by }A_N),$$
we find the average vanishes. Hence
$$\langle F_{A}^-, L_j^N[i_j L_j^NF_A^- \wedge i_jF_A^+]\rangle  = 0,$$
as claimed.
\end{proof}

Now we apply a variant of Lemma \ref{standardcomp} to obtain a commutation result.

\begin{lemma}If $A_N$ holds then 
$$e^*(i_jL_j^NF_A^+) F_{A}^-   = 0.$$
\end{lemma}
\begin{proof}
Once again we let $\psi$ be a smooth compactly supported $ad(E)$ valued $1-$form. 
Then 
$$\|F_A^-\|^2\leq \|F^-_{A+i_j\int_0^t\phi_{j,s}^*F_A^+ds + t^p\psi}\|^2$$
$$= \|F_A^- -2P^-\Phi_j(t) + t^pd_A\psi + t^{2p}\psi\wedge \psi + 2t^p\psi\wedge i_j\int_0^t\phi_{j,s}^*F_A^+ds\|^2$$
$$= \|F_A^-\|^2 -4S_j(t)  + 4t^p\langle F_A^-,\psi\wedge i_j\int_0^t\phi_{j,s}^*F_A^+ds\rangle +O(t^{2N+4}+t^{2p}).$$
Lemma \ref{onion} implies $S_j(t) = O(t^{2N+3})$. 
Hence Taylor expanding again, we get 
$$0\leq  \sum_{b=0}^N4t^p\langle F_A^-,\psi\wedge \frac{t^{b+1}}{(b+1)!}i_jL_j^bF_A^+\rangle +O(t^{2N+3}+t^{2p}) $$
$$=  4t^p\langle F_A^-,\psi\wedge \frac{t^{N+1}}{(N+1)!}i_jL_j^NF_A^+\rangle +O(t^{2N+3}+t^{2p}+t^{N+p+2}).$$
Choosing $p=N+\frac{3}{2}$ gives 
$$0\leq \langle F_A^-,\psi\wedge  i_jL_j^NF_A^+\rangle.$$
Replacing $\psi$ with $-\psi$ makes the inequality an equality, and we conclude 
$$0 = e^*(i_jL_j^NF_A^+)F_A^-$$
as desired. 
\end{proof}

We now require an algebraic proposition. 

\begin{proposition}\label{propinduct}The assumption $A_N$ and the vanishing of $e^*(i_jL_j^NF_A^+)F_A^-$ for all $j$ implies 
$A_{N+1}$. 
\end{proposition}

The proof of this proposition is the goal of the next section.

\subsection{An Algebraic Reduction}
In this subsection, we will prove Proposition \ref{propinduct}. 
\begin{proof}
The Lie derivative $L_j$ differs from the covariant derivative $\nabla_j$ by zero order terms which commute with $ad(E)$. Hence we have 
\begin{equation}\label{algreduce1}
0 = e^*(i_j\nabla_j^NF_A^+)F_A^-.
\end{equation}
Fix a point $x$ and a basis for the infinitesimal isometries so that 
$X_j(x)=e_j$, $j=1,2,3,4$ is an orthonormal basis of $T_xM$.
Then expanding (\ref{algreduce1}) in this frame, yields for all $k$ and $t$, 
\begin{equation}\label{algreduce2}\sum_s[\nabla_k^NF^+_{ks},F^-_{st}] = 0\qquad\text{(no k sum).}\end{equation}
In fact, replacing $e_k$ by $u^1e_1 + \cdots u^4e_4$, we have 
\begin{equation}\label{reducta} \sum_{a,s}[(u^j\nabla_j)^Nu^aF^+_{as},F^-_{st}] = 0.
\end{equation}
Set 
$$p_{ik}(u) = [(u^j\nabla_j)^NF^+_{1i},F^-_{1k}],$$
and let $p_{ik,a} := \frac{\p}{\p u^a}p_{ik}.$
Then we may expand (\ref{reducta}) as 
\begin{align}\label{peq1} 0 &=  u^1(p_{22}+p_{33}+p_{44}) - u^2(p_{34}-p_{43}) + u^3(p_{24}-p_{42}) - u^4(p_{23} - p_{32}).\\
\label{peq2} 0 &=   u^1(p_{34} - p_{43})+ u^2(-p_{22} + p_{44} + p_{33})
- u^3(p_{32} + p_{23}) - u^4(p_{42} + p_{24}).\\
\label{peq3} 0 &=  u^1(-p_{24} + p_{42})- u^2(p_{23} + p_{32})+ u^3(- p_{33} + p_{44}
  + p_{22})  - u^4(p_{43a} + p_{34}).\\
\label{peq4} 0 &=  u^1(p_{23} - p_{32}) -u^2(p_{24} + p_{42}) - u^3(p_{34} + p_{43})+ u^4(-p_{44} + p_{33}
+ p_{22}).\end{align}

In addition to these equations, we have the Yang-Mills equations and the Bianchi identities. These are best encoded as relations between the derivatives of the $p_{ij}$ as follows.  
$$p_{ik,a}(u) = N[(u^j\nabla_j)^{N-1}F^+_{1i,a},F^-_{1k}]
 = N[(u^j\nabla_j)^{N-1}F^+_{1a,i},F^-_{1k}] + N[(u^j\nabla_j)^{N-1}F^+_{ai,1},F^-_{1k}].$$
$$= p_{ak,i}(u) + N[(u^j\nabla_j)^{N-1}F^+_{ai,1},F^-_{1k}].$$
This reduces to $9$ equations. 
$$p_{3k,2}  = p_{2k,3}  + p_{4k,1}.$$
$$p_{4k,3}  = p_{3k,4}  +  p_{2k,1} .$$
$$p_{2k,4}  = p_{4k,2}  +  p_{3k,1}.$$

We may also use  $A_N$ to shift the derivative to the $F_A^-$ term, yielding the following relations among the derivatives. 
$$p_{ik,a} = N[(u^j\nabla_j)^{N-1}F^+_{1i,a},F^-_{1k}] = 
-N[(u^j\nabla_j)^{N-1}F^+_{1i},F^-_{1k,a}]$$
$$ = -N[(u^j\nabla_j)^{N-1}F^+_{1i},F^-_{ak,1}] -N[(u^j\nabla_j)^{N-1}F^+_{1i},F^-_{1a,k}]$$
$$ = N[(u^j\nabla_j)^{N-1}F^+_{1i,1},F^-_{ak}] + p_{ia,k}.$$
This yields $9$ additional equations.

$$p_{24,2}- p_{22,4} = p_{23,1}.$$
$$p_{22,3}-p_{23,2} = p_{24,1}.$$
$$p_{23,4}-p_{24,3} = p_{22,1}.$$
$$p_{34,2}-p_{32,4} = p_{33,1}.$$
$$p_{32,3}-p_{33,2} = p_{34,1} .$$
$$p_{33,4}-p_{34,3} = p_{32,1}.$$
$$p_{44,2}-p_{42,4} = p_{43,1} .$$
$$p_{42,3} - p_{43,2}= p_{44,1} .$$
$$p_{43,4}-p_{44,3} = p_{42,1}  .$$

We get still more equations from considering $d_A^*F_A^{\pm}$. Most but not all of these are dependent on the preceding relations. 

Differentiating (\ref{peq1}) - (\ref{peq4}) in $u^j$, $j= 1,2,3,4$ gives 16 equations. Substituting the above relations into these equations and using the homogeneity relation 
$$p_{ik} = \frac{1}{N}u^ap_{ik,a},$$
allows us to show 
$$p_{ik} = 0,$$
all $i,k$. 
We conclude: for all $i,k$, 
$$0 = [(u^j\nabla_j)^NF^+_{1i},F^-_{1k}].$$
Hence for all $i,k,s,t$
$$0 = [\nabla^NF^+_{si},F^-_{tk}].$$
Finally we have, via application of $A_{N}$, that for $0\leq a\leq N$
$$0 = [\nabla^aF^+_{si},\nabla^{N-a}F^-_{tk}].$$
Hence $A_{N+1}$ holds.\end{proof}
This completes our proof of Proposition \ref{propinduct}. 
\subsection{Splitting} 
Define subsheaves of the sheaf of sections of $ad(E)$ by setting 
$k_x^+\subset ad(E)_x$ (respectively $k_x^-$) to be the subspace generated by the components of $\nabla^aF^+(x)$, (respectively $\nabla^aF^-(x)$) as $a$ runs through all nonnegative integers. By definition, the subsheaves $\mathcal{K}^{\pm}$ of sections of $ad(E)$ which take values in $k^{\pm}$  are preserved by the connection. This gives a reduction of the adjoint bundle.  

Again we let $A$ also denote the local connection form.
Then choosing a gauge with $d_A^*A = 0$ (see, for example, \cite{u1} or \cite{w} for the existence of such gauges) the Yang-Mills equations become a nonlinear elliptic system for $A$. The homogeneous manifolds we are considering are all real analytic, and this system has analytic coefficients. Hence $A$ and $F_A$ are real analytic in this gauge (See \cite{cbm}. See also \cite[Chapter V, Theorem 1.1]{JT} ). 

Fix a point $o$ in $M$, and let $X_1(x),\cdots,X_d(x)$ be Lie polynomials in components of $\nabla_A^aF_A^+(x)$, all $a$,
such that $\{X_1(o),\cdots,X_d(o)\}$ is a basis of $k_o^+$. Then $X_1(x),\cdots,X_d(x)$ are linearly independent in a neighborhood of $o$. Suppose that these vectors do not span $k_x^+$, $x$ near $o$. Then there exists a Lie polynomial, $X_{d+1}$, constructed from the components of the covariant derivatives of $F_A^+$ which, at $x$, is linearly independent of $X_1(x),\cdots,X_d(x)$.  
Then 
$X_1\wedge \cdots\wedge X_d\wedge X_{d+1}$ vanishes to infinite order at $o$. By analyticity, this implies
$X_1\wedge \cdots\wedge X_d\wedge X_{d+1}$ is identically zero in a neighborhood of $o$, contradicting our assumption. Thus we see that $k^+$ and $k^-$ define subbundles of $ad(E)$. This gives the following theorem. 
\begin{theorem}\label{mainth}
Let $E$ be a principal $G-$bundle on a compact oriented homogeneous Riemannian manifold, $M$.
Let $A$ be a Yang-Mills minimizing connection on $E$.
The adjoint bundle, $ad(E)$, contains two $\nabla_A-$stable subbundles, $k^+$ and $k^-$, satisfying 
$F_A^{\pm}$ is a section of $\Lambda^2T^*M\otimes k^{\pm}$,
$$[k^+,k^-] = 0,$$
and the curvature of $k^+$ is self-dual and that of $k^-$ is anti-self-dual.
If $M$ is assumed to be nonnegatively curved instead of compact, then the same result holds provided that we assume that for each $a\geq 0$, $\nabla_A^aF_A\in L^2\cap L^4$.
\end{theorem}\noindent
\begin{proof} The connection $\nabla_A$ preserves $k^{\pm}$ by construction. Hence the curvature of each subbundle is simply the restriction of $F_A$ to the subbundle. As $F_A^{\pm}$ acts trivially on $k^{\mp}$, the curvature operator on $k^{\pm}$ is $F_A^{\pm}$ yielding the asserted self-duality and anti-self-duality.\end{proof} 

\begin{corollary}Let $M$ be a compact homogeneous 4-manifold. Let $E$ be a principal $G-$bundle over $M$ with Yang-Mills minimizing connection $A$. Suppose the first Pontrjagin number of $E$ is greater than or equal to zero (respectively less than or equal to zero). Then if the Yang-Mills energy of $A$ is greater than the topological lower bound determined by the first Pontrjagin number of $E$, $ad(E)$ has a nontrivial subbundle with anti-self-dual curvature (respectively self-dual curvature).
\end{corollary}
\section{Dimension 3}
We now consider applications of the variational inequality to three dimensions. 
\begin{theorem}\label{th3}Let $Y$ be a complete three dimensional manifold with nonnegative Ricci curvature. Let $A$ be a Yang-Mills minimizing connection on a $G$-bundle $E$ on $Y$. If $M$ is noncompact, assume further that
$F_A\in L^4$. Then 
$$\nabla_A F_A = 0,$$
and 
$$Ric(\ast F_A,\ast F_A) = 0.$$
Moreover, $F_A$ takes values in a flat commutative subbundle of $ad(E)$, 
and $F_A=0$ unless $M$ has local flat factors. 
\end{theorem}
\noindent
\begin{proof} The variational inequality gives 
$$\|F_A\|^2 \leq \|F_{A + t\ast F_A}\|^2 = \|F_A + t^2\ast F_A\wedge \ast F_A\|^2.$$
Hence,
\begin{equation}\label{var3d}0\leq \langle F_A, \ast F_A\wedge \ast F_A\rangle.\end{equation}
 
On the other hand combining the Yang-Mills equation and the Bochner formula, we have 
$$0 = \|d_A F_A\|^2 + \|d_A^*F_A\|^2 = \|\nabla_A F_A\|^2 - \langle \sum_{ij}a_ia_j^*R_{ij}F_A,F_A\rangle 
- \langle \sum_{ij}a_ia_j^*[F_{ij},F_A],F_A\rangle $$
$$ = \|\nabla_A F_A\|^2 + \int_Y Ric (\ast F_A  , \ast F_A)dv + 2\langle F_A, (\ast F_A)\wedge (\ast F_A)\rangle. $$
Here $R$ denotes the Riemann curvature.  
 
Thus  if the Ricci curvature of $Y$ is nonnegative, we conclude from this Bochner formula and inequality (\ref{var3d}) that
$$0 = \langle F_A, (\ast F_A)\wedge (\ast F_A)\rangle,$$
and 
\begin{equation}\label{const}\nabla_A F_A = 0.\end{equation}
If the Ricci curvature is strictly positive at some point, then $F_A=0$.
Equation (\ref{const}) implies that the subbundle $H$ of $ad(E)$ generated by the components of $F_A$ is stable under $\nabla_A$.
 
Applying Lemma \ref{standardcomp} with $P^- = 1 $ and $\psi = \ast F_A$, we deduce
$$0 = [F_{ij},F_{st}],$$
for all $i,j,s,t.$
Hence $H$ is a commutative flat subbundle of $ad(E)$. 
Now $\nabla_AF_A = 0 $ implies 
the Riemannian curvature acts trivially on the subbundle of $T^*M$ determined by $F_A$. 
Thus $M$ has local flat factors unless $F_A = 0.$

\end{proof}

 Up to diffeomorphism, the only simply connected three manifold with strictly positive Ricci curvature is $S^3$ (see \cite{h}). Uhlenbeck's compactness theorem (\cite{u1}) implies that every $G-$bundle on a compact 3 manifold has a Yang-Mills minimizing $G-$connection. Theorem \ref{th3} then implies the minimizing connection is flat on $S^3$ and its smooth finite quotients. On $S^3$ it is therefore trivial as a $G-$bundle. As 
$G$-bundles on $S^3$ are classified by $\pi_2(G)$, this gives an analytic proof of the well known fact that $$\pi_2(G) = 0$$ for all compact connected Lie groups. (See \cite[Section 18]{bo}). We similarly deduce that all $G-$bundles on $T^3$ admit $G-$connections with covariant constant curvature.

\section{Calabi--Yau 3 folds} 
Let $M$ be a compact Calabi Yau 3 fold, with Kahler form $\omega$ and nonzero covariant constant $(3,0)$ form $\Omega.$  Let $A$ be a minimizing connection.  

Decompose the curvature, $F_A$ as 
$$F_A = F_A^{2,0} + F^{1,1}_{A0} + \phi_A\omega +F_A^{0,2},$$
where 
$$\phi_A(x) = \frac{1}{3}\langle F_A,\omega\rangle (x).$$
Recall that the Bianchi equations imply 
\begin{equation}\label{cybianch}0 = \db_A F_{A0}^{1,1} + \db_A\phi_A\wedge\omega + \partial_A F^{0,2}_A.\end{equation}
Let $\Lambda = e^*(\omega)$.  The Yang-Mills equations coupled to the Kahler identities imply 
\begin{equation} \label{cyym}0 = -i\Lambda \db_A F_{A0}^{1,1} - i\Lambda (\db_A\phi_A\wedge\omega) + 3i\db_A\phi_A + i\Lambda\partial_A F^{0,2}_A.\end{equation}
Combining (\ref{cybianch}) and (\ref{cyym}) gives first that 
$$\db_A^*F^{0,2}_A = -3\db_A\phi_A.$$
Taking inner products, we have 
$$\|\db_A^*F^{0,2}_A\| = -3\langle \db_A^*F^{0,2}_A,\db_A\phi_A\rangle = 
-3\langle F^{0,2}_A,\db_A^2\phi_A\rangle = -3Re\langle F^{0,2}_A,[F^{0,2}_A,\phi_A]\rangle = 0.$$
Here we are using that $\phi_A$ is real and $Re[\bar F_{st}^{0,2},F_{st}^{0,2}] = 0.$ 
Hence 
$$d_A^*F_A^{0,2} = 0 = \nabla_A\phi_A.$$
Thus the splitting of $ad(E)$ into eigenspaces of $ad(\phi_A)$ is preserved by the connection.

Define an $ad(E)$ valued one form $\psi$ so that 
\begin{equation}\label{defpsi}e^*(\psi)\bar \Omega = F_A^{0,2}.
\end{equation} 
Applying $d_A^*$ to each side of (\ref{defpsi}) gives 
$$e^*(\db_A\psi)\bar \Omega = 0,$$ 
and therefore 
$$\db_A\psi = 0.$$
Applying $d_A$ to both sides of (\ref{defpsi}) gives  
$$\db_A^*\psi = 0.$$

On a Calabi--Yau, we may reexpress the Yang-Mills energy as 
$$\|F_A\|^2 = -\int_M tr F_A\wedge \ast F_A
 = - \int_M tr (-F_{A0}^{1,1}\wedge F_{A0}^{1,1} + 2F_A^{2,0}\wedge F_A^{0,2} + \phi_A\omega\wedge \phi_A \omega]\wedge \omega$$
$$=  4\|F_A^{0,2}\|^2 + 2\|\phi_A\omega\|^2  + \int_M tr(F_A\wedge F_A)\wedge\omega.$$ 

The last term is independent of $A$. Hence, if we define $P^-$ to be the projection onto the $(0,2)+(2,0)+ \langle\omega\rangle$ summands, then we see that it defines a conservative decomposition. We therefore have the consequent additional variational inequalities.

Consider the variation $A+t(\psi+ \bar\psi)$ to get 
$$0\leq 2Re\langle F_A^{0,2},\psi\wedge \psi\rangle +  \frac{2}{3}\|\Lambda d_A(\psi + \bar\psi)\|^2  + \frac{4}{3}Re\langle (\Lambda F_A)\omega,(\psi + \bar\psi)\wedge (\psi + \bar\psi)\rangle$$
$$= 2Re\langle F_A^{0,2},\psi\wedge \psi\rangle  + 8Re\langle \phi_A\omega, \psi \wedge \bar\psi\rangle $$
$$= 2Re\langle F_A^{0,2},\psi\wedge \psi\rangle.$$
Here we have used 
$$\nabla_A\phi_A = 0\Rightarrow [F_A,\phi_A]= 0\Rightarrow [\bar \psi,\phi_A] = 0.$$

Considering instead the variation 
$A+it(\psi- \bar\psi)$, we obtain 
$$0\leq  -2Re\langle F_A^{0,2},\psi\wedge \psi\rangle.$$
Hence the inequalities are equalities, and we may deduce from Lemma \ref{standardcomp} that 
$$e^*(\psi)F_A^{0,2} = 0.$$
In components, this is equivalent to 
\begin{equation}\label{comcy}0 = [F_{st}^{0,2},F_{ab}^{0,2}],\end{equation}
all $s,t,a,b.$    
The components of $F_A^{0,2}$  thus generate an abelian subalgebra $\mathfrak{a}$ of $ad(E)\otimes \IC.$ 
We summarize the computations in this section with a theorem.
\begin{theorem}\label{cyth}
Let $A$ be a smooth minimizing Yang Mills connection on a Calabi Yau threefold. If $M$ is noncompact, assume further that  $ F_A\in L^4$. Then $F_A^{0,2}$ takes values in a commutative subbundle of $ad(E)\otimes \IC.$ 
\end{theorem}

%
 \newcommand{\etalchar}[1]{$^{#1}$}
\def\polhk#1{\setbox0=\hbox{#1}{\ooalign{\hidewidth
  \lower1.5ex\hbox{`}\hidewidth\crcr\unhbox0}}}
\providecommand{\bysame}{\leavevmode\hbox to3em{\hrulefill}\thinspace}

%
\end{document}